\documentclass[12pt]{article}
\input{amssym}
\newtheorem{obs}{Observation}[section]
\newtheorem{thm}[obs]{Theorem}
\newtheorem{cor}[obs]{Corollary}

\newtheorem{prop}[obs]{Proposition}
\newtheorem{lem}[obs]{Theorem}
\newtheorem{deff}[obs]{Definition}
\newtheorem{conj}[obs]{Conjecture}

\newtheorem{preproof}{{\bf Proof.}}

\newenvironment{proof}[1]{\begin{preproof}{\rm
               #1}\hfill{$\rule{2mm}{2mm}$}}{\end{preproof}}

\def\n#1{\vbox to 3mm{\vspace{1mm}\vfill \hbox to 2.0mm{\hfill
             $#1$\hfill} \vfill }}
\def\m#1#2{\raise 0.2ex\hbox{
    ${#1_{\bf \displaystyle #2}}$}}
\def\x#1{\raise 0.5ex\hbox{
    ${#1}$}}
\usepackage{array}
\title{{\bf Small oriented cycle double cover of graphs}}
{\small
\author{
{\sc Behrooz Bagheri Gh.} and
{\sc Behnaz Omoomi}\\
[1mm]
{\small \it  Department of Mathematical Sciences}\\
{\small \it  Isfahan University of Technology} \\
{\small \it 84156-83111,\ Isfahan, Iran}}
\date{}

\begin{document}
\maketitle
%
%
%
%
%
%
%
%
\begin{abstract}
 A { small oriented cycle double cover (SOCDC)} of a
bridgeless  graph $G$ on $n$ vertices is a collection of at most $n-1$ directed cycles of the symmetric
orientation, $G_s$, of $G$
 such that each arc of  $G_s$ lies in exactly one of the cycles. It is conjectured that every
 $2$-connected graph except two complete graphs $K_4$ and $K_6$ has an $\rm SOCDC$.
 In this paper, we study graphs with $\rm SOCDC$ and obtain some properties of the minimal counterexample to this conjecture.
\end{abstract}
\section{\large Introduction}        
We denote by $G$ a finite undirected graph  with vertex set $V$ and  edge set $E$ with no loops or
multiple edges. The {\sf symmetric orientation} of $G$, denoted
by $G_s$,  is an oriented graph obtained from $G$ by replacing
each edge of $G$ by a pair of opposite directed arcs.
An  {\sf  even graph (odd graph)} is a graph
 such that each vertex is incident to an even (odd) number of edges.
A {\sf directed even graph}  is a graph  such that
for each vertex its out-degree equals to its in-degree.
A {\sf cycle} ({\sf a directed cycle}) is a minimal non-empty even graph (directed even graph).
We denote every directed cycle $C$  and directed path $P$ on $n$ vertices with vertex set $\{v_1,\ldots,v_n\}$
 and directed edge set $E(C)=\{v_iv_{i+1},v_nv_1 : 1\le i\le n-1\}$ and  $E(P)=\{v_iv_{i+1} : 1\le i\le n-1\}$  by $C=[v_1,\ldots,v_n]$, and $P=(v_1,\ldots,v_n)$, respectively.

A {\sf cycle double cover (CDC)} $\mathcal{C}$ of a graph $G$
is a collection of cycles in $G$ such that every edge of $G$
belongs to exactly two cycles of $\mathcal{C}$. Note that the
cycles are not necessarily distinct. It can be easily seen that
a necessary condition for a graph to have a $\rm CDC$ is that
the graph has no cut edge which is called a bridgeless graph. Seymour~\cite{Sey} in $1979$ conjectured
that every bridgeless graph has a $\rm CDC$. No counterexample to the $\rm CDC$
conjecture is known. It is proved that the minimal counterexample to the $\rm CDC$
conjecture is a bridgeless cubic graph with edge chromatic number equal to $4$, which is called a {\sf snark}.

A {\sf small cycle double cover (SCDC)} of a graph on $n$ vertices
is a $\rm CDC$ with at most $n-1$ cycles.
There exist simple graphs of order $n$  for which any $\rm CDC$
requires at least $n-1$ cycles (e.g., $K_n, n\ge 3$). Furthermore,
no simple bridgeless graph of order $n$  is known to require more
than $n-1$ cycles in a $\rm CDC$. Note that clearly it is false
if not restricted to simple graphs. Bondy~\cite{Bondy}  conjectured that
every simple bridgeless graph  has an $\rm SCDC$. For more results on the $\rm CDC$ conjecture see~\cite{JaegerCDC,Zhang}.

The $\rm CDC$ conjecture has many stronger forms. In this paper,
we consider the oriented version of these conjectures.

An {\sf oriented cycle double cover (OCDC)} is a $\rm CDC$ in which every cycle can be
 oriented in such a way that every edge of the graph is covered by
two directed cycles in two different directions.
\begin{conj}\label{OCDC}{\rm~\cite{JaegerNZF}} {\rm ($\rm Oriented\ CDC$ conjecture)}
Every bridgeless graph has an $\rm OCDC$.
\end{conj}

No counterexample to this conjecture is known.
It is clear that  the validity of the $\rm OCDC$ conjecture implies the validity of the $\rm CDC$ conjecture.
While there is a $\rm CDC$ of the Petersen graph that can not be oriented  in such a way that forms an $\rm OCDC$.
\begin{deff}
A {\sf small oriented cycle double cover (SOCDC)} of a graph
on $n$ vertices is an $\rm OCDC$ with at most $n-1$ directed cycles.
\end{deff}

A {\sf perfect path double cover (PPDC)}  of a graph $G$
 is a collection $\mathcal{P}$  of paths in $G$ such that each
edge  of $G$ belongs  to  exactly two  members of $\mathcal{P}$
and  each vertex  of $G$ occurs  exactly twice as an  end  of a
path in  $\mathcal{P}$~\cite{BondyPPDC}. In~\cite{PPDC} it is
proved that every simple graph has a $\rm PPDC$.

An {\sf oriented perfect path double cover  (OPPDC)} of a
graph $G$ is a collection of directed paths in the symmetric
orientation $G_s$ such that each arc of $G_s$ lies in exactly
one of the paths and  each vertex  of $G$ appears just once as
a beginning and just once as end of a directed path.
Maxov{\'a} and  Ne{\v{s}}et{\v{r}}il in~\cite{MR1825628} showed that
two complete graphs $K_3$ and $K_5$ have no  $\rm OPPDC$ and in~\cite{MR2046642},
they conjectured every connected graph except $K_3$ and $K_5$ has an $\rm OPPDC$.

The {\sf join} of two simple graphs $G$ and $H$, $G\vee H$, is the graph obtained from
the disjoint union of $G$ and $H$ by adding the edges $\{uv :  u\in V(G),\ v\in V(H)\}$.

The existence of a ${\rm PPDC}$ for
graphs in general is equivalent to the existence of an ${\rm SCDC}$ for the bridgeless
graph obtained by joining a new vertex to all other vertices~\cite{BondyPPDC}. The following theorem denotes a relation between $\rm OPPDC$ and $\rm SOCDC$.
\begin{lem}\label{n-1}~{\rm \cite{MR1825628}}
Let $G$ be a connected graph.  The graph $G$ has an $\rm OPPDC$
 if and only if $G\vee K_1$ has an $\rm SOCDC$.
\end{lem}
%

In the following theorem a list of some families of
graphs that admit an $\rm OPPDC$   is provided. Therefore by Theorem~$\ref{n-1}$,
the join of  graphs satisfying at least one of the conditions in below and $K_1$ admit an $\rm SOCDC$.
\begin{lem}~{\rm\cite{BaOm:OPPDC,MR1825628}}\label{main} Let $G\ne K_3$ be a graph.
In  each of the following cases, $G$ has an $\rm OPPDC$.
\begin{description}
\item{\rm\bf(i)} $G$ is a union of two arbitrary trees.
\item{\rm\bf(ii)} $G$ is an odd graph.
\item{\rm\bf(iii)} $G$ has no adjacent vertices of degree greater than two.
\item{\rm\bf(iv)} $G$ is a $2$-connected graph of order $n$ and $|E(G)|\le 2n-1$.
\item{\rm\bf(v)} $G=L(T)$, for some tree $T$.
\item{\rm\bf(vi)} $G=L(H)$, where the degree of no adjacent vertices in $H$
have the  same parity.
\item{\rm\bf(vii)}  $G$ is a graph with $\Delta(G)\le 4$ and $\delta(G)\le 3$.
\item{\rm\bf(viii)} $G$ is a  separable $4$-regular graph. {\rm (}A separable graph is a graph contains cut vertex.{\rm )}
\end{description}
\end{lem}

In what follows we have three sections. Section~\ref{sec:SOCDC} deals with certain families of graphs with a small oriented cycle double cover.
It is conjectured that every
 $2$-connected graph except two complete graphs $K_4$ and $K_6$ has an $\rm SOCDC$.
In Section~\ref{sec:MCOSCDC}, we  study the properties of the minimal counterexample to this conjecture.
Finally in Section~\ref{sec:CP},
some more relations between $\rm OPPDC$ and $\rm SOCDC$ are given.
\section{The small oriented cycle double cover}\label{sec:SOCDC}
The natural question is that which simple bridgeless graphs of
order $n$ have an $\rm OCDC$ with at most $n-1$ cycles ($\rm SOCDC$)? \\

Since $K_3$ and $K_5$ have no $\rm OPPDC$, by Theorem~\ref{n-1},
$K_4$ and $K_6$ have no $\rm SOCDC$.
It is known that every $K_{2n-1}, n\ge 4$,
has an $\rm OPPDC$~\cite{BaOm:OPPDC}, thus by Theorem~\ref{n-1}, every $K_{2n},\ n\ge 4$, has  an $\rm SOCDC$.
Moreover, every $K_{2n+1}$ has an $\rm SOCDC$, since $K_{2n+1}$ has a Hamiltonian cycle decomposition~\cite{West}.

The following observation shows that if every block of a graph $G$ has an $\rm SOCDC$,
then $G$ has also an $\rm SOCDC$.
\begin{obs}\label{cut-vertex}
If $G=G_1 \cup G_2$ and $V(G_1) \cap V(G_2) = \{v\}$ which $G_i$
is a graph  with an $\rm SOCDC$,  $i=1,2$; then $G$ also has an $\rm SOCDC$.
\end{obs}
%
 %
%

Moreover, Observation~\ref{cut-vertex} directly concludes the following corollaries.
A {\sf block graph} is a graph for which each block is a clique.
\begin{cor}
Every block graph with no block of order $2,4$ and $6$ has an $\rm SOCDC$.
\end{cor}

Since the line graph of every tree is a block graph, the following result obtained which
 is an oriented version of existence of $\rm SCDC$
of line graph of trees~\cite{lineSCDC}.
\begin{cor}
If $T$ is a tree without vertices of degree $2,4$ or $6$, then $L(T)$ has an $\rm SOCDC$.
\end{cor}

In the following proposition, we construct some graphs with no $\rm SOCDC$.
In fact, we show that the difference $|\mathcal{C}|-(n-1)$
 could be large enough for every $\rm OCDC$, $\mathcal{C}$ of some bridgeless graph of order $n$.

Let $V(K_4)=\{v_1,v_2,v_3,v_4\}$. The collection $\mathcal{C}=\{[v_{1},v_{2},v_{4}],[v_{2},v_{1},v_{3}],[v_{3},v_{4},v_{2}],$\\$[v_{4},v_{3},v_{1}]\}$ is an $\rm OCDC$ of $K_4$.
%
Since $K_4$ has six edges, if $\mathcal{C}$
is an arbitrary $\rm OCDC$ of $K_4$, then $|\mathcal{C}|\le (2\times6)/3=4$.
Thus, every $\rm OCDC$ of $K_4$ is of size $4$.

Let $V(K_6)=\{v_1,\ldots,v_6\}$. The collection
$\mathcal{C}=\{[v_1,v_2,v_3,v_4,v_5,v_6],[v_2,v_6,v_3,v_5,v_4],$\\$[v_1,v_5,v_2,v_4,v_3],[v_1,v_4,v_6,v_2,v_5],[v_1,v_6,v_5,v_3,v_2],[v_1,v_3,v_6,v_4]\}$
is an $\rm OCDC$ of $K_6$ of size 6.
\begin{prop}
For every integer $r\ge 1$, there exists a bridgeless graph $G$ of order $n$
such that every $\rm OCDC$  of $G$ has  $(n-1)+r$  directed cycles.
\end{prop}
\begin{proof}{
Let $P$ be a path of length $r$ with $V(P)=\{v_1,\ldots,v_{r+1}\}$
and $E(P)=\{v_iv_{i+1} : 1\le i\le r\}$.
Assume that $G$ is a graph obtained from $P$ by replacing each
edge $v_iv_{i+1}$ of $P$ with a clique $K_4$, say $K_4^i$, where
$V(K_4^i)=\{v_i,v_i',v_{i+1},v_{i+1}'\}$, $1\le i\le r$.
Every $\rm OCDC$
of $G$ is decomposable to $r$ $\rm OCDC$ of $K_4$. Moreover, every $\rm OCDC$ of $K_4$
has four cycles.
Therefore, every $\rm OCDC$ of $G$ has $4r$ cycles. Note that $|V(G)|=3r+1$,
thus every $\rm OCDC$ of $G$ has  $(|V(G)|-1)+r$ cycles.
}\end{proof}

This fact motivates us to present the following conjecture.
\begin{conj}\label{SOCDCconj}~{\rm (SOCDC conjecture)}
Every simple $2$-connected graph except $K_4$ and $K_6$ admits an $\rm SOCDC$.
\end{conj}

The above conjecture has a close relation to the following conjecture.
\begin{conj}{\rm~\cite{projective}} {\rm (Haj$\rm \acute{o}$s' conjecture)}
 If $G$ is a simple, even graph of order $n$, then $G$ can be
decomposed into $\lfloor(n-1)/2\rfloor$ cycles.
\end{conj}

If the Haj$\rm \acute{o}$s' conjecture holds, then every even graph has
an $\rm SOCDC$ obtained by taking two copies of the cycles used in
its decomposition, in two opposite directions.

An edge of a graph  $G$ is said to be contracted if it is
deleted and its two ends are identified. A {\sf minor} of $G$ is
a graph obtained  from $G$ by deletions  of vertices,  and deletions
and contractions of edges. The graph obtained
from $K_6$ by deleting an edge is denoted $K_6^-$.
A {\sf $K_6^-$-minor free} graph is a graph  that does not contain $K_6^-$  as a minor.

As the Haj$\rm \acute{o}$s' conjecture is true for even graphs with
maximum degree four~\cite{Favaron}, planar graphs~\cite{Seyf},
projective graphs (a projective graph is a graph  $G$
which is embeddable on the projective plane.), and $K_6^-$-minor free graphs~\cite{projective},
these graphs have an $\rm SOCDC$.
\begin{prop}
Let $G$ be an even graph. In  each of the following cases, $G$ has an $\rm SOCDC$.
\item{\rm\bf(i)} $\Delta(G)=4$.
\item{\rm\bf(ii)} $G$ is planar.
\item{\rm\bf(iii)}  $G$ is a projective graph.
\item{\rm\bf(iv)} $G$ is $K_6^-$-minor free.
\end{prop}

Klimmek~\cite{Klim} proved that every even line graph of order $n$
has a cycle decomposition into $\lfloor(n-1)/2\rfloor$ cycles,
thus the Haj$\rm \acute{o}$s' conjecture  holds for  such graphs. Since
a line graph, $L(G)$, is even if and only if  every component of $G$ is
 either even or odd, the line graph of every even graph
and  of every odd  graph has an $\rm SOCDC$.
\begin{prop}
If $G$ is an even or an odd graph, then $L(G)$ has an $\rm SOCDC$.
\end{prop}

The following proposition considers another class of graphs with
$\rm OCDC$ which also has $\rm SOCDC$.
\begin{prop}
If $G$ has an $\rm OCDC$, $\mathcal{C}$, and the girth
of $G$, $g(G)$, is greater than average degree, $\bar{d}(G)$, then $\mathcal{C}$
is also an $\rm SOCDC$ of $G$.
\end{prop}
\begin{proof}{ Let $\mathcal{C}$ be an $\rm OCDC$ of $G$.
Note that each edge of  $G$ is covered twice by elements of $\mathcal{C}$, therefore,
 $$ g(G)|\mathcal{C}|\le \sum_{C\in \mathcal{C}}|E(C)|=2|E(G)|=\sum_{v\in V(G)}d(v)= |V(G)|\bar{d}(G).$$
 Since $g(G)> \bar{d}(G)$, we have $|\mathcal{C}| \le |V(G)|-1$. Hence, $\mathcal{C}$
is an $\rm SOCDC$ of $G$.
 }\end{proof}

%

It can be proved that an $\rm OCDC$ for planar graphs can be obtained from their planar embedding
and some planar graph has also $\rm SOCDC$.
%
%
%
\begin{prop}\label{planar}
Every bridgeless planar graph $G$ with $|E(G)| < 2|V(G)|-2$, has an $\rm SOCDC$.
\end{prop}
\begin{proof}{
Let $G$ be a bridgeless planar graph. 
Since we can orient the edges of each face of $G$ in such a way that the collection of the boundary of its faces,
$\mathcal{F}$, is an $\rm OCDC$. By Euler's formula, $|\mathcal{F}|=2+|E(G)|-|V(G)|$.
Since  $|E(G)| < 2|V(G)|-2$, we conclude $|\mathcal{F}| < |V(G)|$.
Hence, $G$ has an $\rm SOCDC$.
}\end{proof}

Since in every simple triangle-free  planar graph $G$ with at least three vertices,
 $|E(G)| \le 2|V(G)|-4$, we obtain the following corollary.
%
\begin{cor}
Let $G$ be a bridgeless planar graph. If $G$ is triangle-free,
then $G$ admits an $\rm SOCDC$.
\end{cor}

The following proposition presents an $\rm SOCDC$ for the well-known non-planar triangle-free graphs.
\begin{prop}
Every $K_{n,m},\ n,m\ge 2,$ has an $\rm SOCDC$.
\end{prop}
\begin{proof}{
Assume that  $V(K_{n,m})=\{v_1,\ldots,v_n;w_1,\ldots,w_m\}$,  $n\le m$.
Let $$C_i=[v_1,w_{i},v_2,w_{i+1},v_3,w_{i+2},\ldots,v_{n-1},w_{i+n-2},v_n,w_{i+n-1}],$$
be a directed cycle, where subscripts are reduced modulo $m$.
It is easy to check that $\mathcal{C}=\{C_i : 1\le i\le m\}$ is an $\rm SOCDC$
of $K_{n,m},\ n,m\ge 2$.
}\end{proof}
%

Let $G$ be a simple graph and $(D,f)$ be an ordered pair where $D$ is
an orientation of $E(G)$ and $f$ is a weight on $E(G)$ to $\Bbb{Z}$.
For each $v\in V(G)$, denote
$$f^+(v)=\sum{f(e)}\ \ \ \ \ \mbox{and}\ \ \ \ \ f^-(v)=\sum{f(e)},$$
where the summation is taken over all directed edges of $G$ (under the orientation $D$)
with tails and heads, respectively, at the vertex $v$.
An {\sf integer flow} of $G$ is an ordered pair $(D,f)$ such that  for every vertex $v\in V(G)$, $f^+(v)=f^-(v)$.
 A {\sf nowhere-zero $k$-flow} of $G$ is an integer flow $(D,f)$ such that $0<|f(e)|<k$, for every edge  $e\in E(G)$ and is denoted by $k$-$\rm NZF$~{\rm \cite{Zhang}}.

\begin{lem}~{\rm\cite{Zhang}}\label{zhang:X'=3}
Every cubic graph $G$ admits a $4$-$\rm NZF$ if and only if  $\chi'(G)=3$.
\end{lem}
\begin{lem}~{\rm\cite{Zhang}}\label{zhang:O4CDC}
A graph $G$ admits a $4$-$\rm NZF$ if and only if
$G$ has an $\rm OCDC$ consists of four directed even subgraphs.
\end{lem}

 The following theorem concludes from Theorems~\ref{zhang:X'=3} and~\ref{zhang:O4CDC}.

\begin{lem}\label{cubicX'3}
Every cubic graph with edge chromatic number $3$ admits an $\rm OCDC$.
\end{lem}
\begin{lem}\label{n/2}{\rm~\cite{cubic}}
If $\mathcal{C}$ is a $\rm CDC$ of a cubic graph $G$ of order $n$, then $|\mathcal{C}|\le {n}/{2}+2$.
\end{lem}

Since from every $\rm OCDC$ of a graph a $\rm CDC$ for the graph is obtained, we have the following corollary.

\begin{cor}\label{OCDCcubic}
Every $\rm OCDC$, $\mathcal{C}$, of a cubic graph of order $n\ge 6$, is an $\rm SOCDC$.
\end{cor}

The following corollary  concludes directly from Theorem~\ref{cubicX'3}
and Corollary~\ref{OCDCcubic}.
\begin{cor}
Every cubic graph with edge chromatic number $3$, $G\ne K_4$, has an $\rm SOCDC$.
\end{cor}

\section{\large The minimal counterexample to the $\rm SOCDC$ conjecture}\label{sec:MCOSCDC}
If the $\rm CDC$ conjecture is false, then it must have a minimal counterexample.
In this section, we study the properties of the minimal counterexample to the $\rm SOCDC$ conjecture.
\begin{obs}\label{subdivide}
If $G$ is a graph with an $\rm SOCDC$ and $G'$ is the graph obtained from $G$
by dividing one edge of $G$, then $G'$ also admits an $\rm SOCDC$.
\end{obs}
%
%
\begin{cor}
 Let $G$  be the minimal counterexample to the $\rm SOCDC$ conjecture, then the minimum degree of $G$ is at least $3$.
\end{cor}
\begin{thm}\label{SOCDCK3}
 The minimal counterexample to the $\rm SOCDC$ conjecture is $3$-connected.
\end{thm}
\begin{proof}{

Let $G$, the minimal counterexample to the $\rm SOCDC$ conjecture be a $2$-connected graph of order $n$ with vertex cut $\{v_1,v_2\}$ and
 $G=G_1\cup G_2$, where $V(G_1)\cap V(G_2)=\{v_1,v_2\}$ and
$|V(G_i)|=n_i,\ i=1,2$.
 Assume that $G_i\cup\{v_1v_2\}$ has an $\rm SOCDC,\ \mathcal{C}_i,\ i=1,2$.
Let $C_i^j, j=1,2,$ be the two directed cycles in $\mathcal{C}_i$, $i=1,2,$ which
include the directed edge $v_jv_{j+1}$,
where subscripts are reduced modulo~$2$. In each of the following cases, we show that $G$ admits an $\rm SOCDC$, which is a contradiction.
\begin{description}
\item(I) If $v_1v_2\in E(G)$, then we define
$$\mathcal{C}=\mathcal{C}_1\cup \mathcal{C}_2\cup \{C_1^1\Delta C_2^2\}\setminus \{C_1^1, C_2^2\}.$$
The collection $\mathcal{C}$ is an $\rm OCDC$ of $G$, where
$$|\mathcal{C}|=|\mathcal{C}_1|+|\mathcal{C}_2|-1\le (n_1-1)+(n_2-1)-1$$
$$\hspace{18.2mm}\le (n_1+n_2)-3$$
$$\hspace{30mm}\le (n+2)-3=n-1.$$

If $G_1\cup\{v_1v_2\}=K_4$ with $V(K_4)=\{v_1,v_2,v_3,v_4\}$,  and $G_2\cup \{v_1v_2\}$ has an $\rm SOCDC$, say $\mathcal{C}_2$, then let $C_1=[v_1,v_2,v_4]$,  $C_2=[v_1,v_4,v_3,v_2]$,
$C_3=C_2^1\cup (v_1,v_3,v_4,v_2) \setminus \{v_1v_2\}$, and $C_4=C_2^2\cup (v_2,v_3,v_1) \setminus \{v_2v_1\}$.
Therefore, $$\mathcal{C}=\mathcal{C}_2\cup \{C_1,C_2,C_3,C_4\}\setminus \{C_1^1, C_2^2\}$$ is an $\rm SOCDC$ of $G$.

If $G_1\cup\{v_1v_2\}=K_6$ with $V(K_6)=\{v_1,v_2,v_3,v_4,v_5,v_6\}$,  and $G_2\cup \{v_1v_2\}$ has an $\rm SOCDC$, say $\mathcal{C}_2$, then let $C_1=[v_1,v_2,v_4,v_6,v_3,v_5]$,
 $C_2=[v_1,v_3,v_6,v_2]$, $C_3=[v_1,v_4,v_2,v_5,v_6]$, $C_4=[v_1,v_5,v_2,v_3,v_4]$,
$C_5=C_2^1\cup (v_1,v_6,v_5,v_4,v_3,v_2) \setminus \{v_1v_2\}$, and $C_6=C_2^2\cup (v_2,v_6,v_4,v_5,v_3,v_1) \setminus \{v_2v_1\}$.
Therefore, $$\mathcal{C}=\mathcal{C}_2\cup \{C_1,C_2,C_3,C_4,C_5,C_6\}\setminus \{C_1^1, C_2^2\}$$ is an $\rm SOCDC$ of $G$.

If $G_1\cup\{v_1v_2\}=G_2\cup\{v_1v_2\}=K_4$ or $G_1\cup\{v_1v_2\}=K_4$ and  $G_2\cup\{v_1v_2\}=K_6$ or $G_1\cup\{v_1v_2\}=G_2\cup\{v_1v_2\}=K_6$, then 
by Theorem~1 in~\cite{BaOm:OPPDC}, $G\setminus v_1$ admits an $\rm OPPDC$, thus by Theorem~\ref{n-1}, $G$ has an $\rm SOCDC$.

\item(II) If $v_1v_2\notin E(G)$, then we define
$$\mathcal{C}=\mathcal{C}_1\cup \mathcal{C}_2\cup \{C_1^1\Delta C_2^2,C_1^2\Delta C_2^1\}\setminus \{C_1^1, C_1^2,C_2^1, C_2^2\}.$$
The collection $\mathcal{C}$ is an $\rm OCDC$ of $G$, where $|\mathcal{C}|\le n-2$.\\
Furthermore, if $G_1\cup \{v_1v_2\}=K_4$ or $K_6$,  $v_1v_2\notin E(G)$, and
$G_2\cup \{v_1v_2\}$ has an $\rm SOCDC$, by the similar argument in above using the given $\rm SOCDC$ for $K_4$ and $K_6$
of size $4$ and $6$, an $\rm SOCDC$ for $G$ is obtained.

If $G_1\cup\{v_1v_2\}=G_2\cup\{v_1v_2\}=K_4$ with $V(G_1)=\{v_1,v_2,v_3,v_4\}$ and $V(G_2)=\{v_1,v_2,v_5,v_6\}$, then 

$\mathcal{C}=\{[v_1,v_4,v_3,v_2,v_5,v_6],[v_1,v_5,v_2,v_3],[v_1,v_3,v_4,v_2,v_6,v_5],[v_1,v_6,v_2,v_4]\}$ 

is an $\rm SOCDC$ of $G$.

If $G_1\cup\{v_1v_2\}=K_4$ with $V(G_1)=\{v_1,v_2,v_3,v_4\}$ and $G_2\cup\{v_1v_2\}=K_6$ $V(G_2)=\{v_1,v_2,v_5,v_6,v_7,v_8\}$, then

$\mathcal{C}=\{[v_1,v_6,v_5,v_7,v_8,v_2,v_3],[v_1,v_3,v_4,v_2,v_8],[v_1,v_7,v_6,v_8,v_5,v_2,v_4],[v_1,v_5,v_8,$

\hspace{1cm}$v_7,v_2,v_6],
[v_1,v_8,v_6,v_2,v_7,v_5],[v_1,v_4,v_3,v_2,v_5,v_6,v_7]\}$ 

is an $\rm SOCDC$ of $G$.

If $G_1\cup\{v_1v_2\}=G_2\cup\{v_1v_2\}=K_6$ with $V(G_1)=\{v_1,v_2,v_3,v_4,v_5,v_6\}$ and  $V(G_2)=\{v_1,v_2,v_7,v_8,v_9,v_{10}\}$, then

$\mathcal{C}=\{[v_1,v_6,v_4,v_5,v_3,v_2,v_7,v_9,v_8,v_{10}],[v_1,v_3,v_5,v_4,v_6,v_2,v_{10},v_8,v_9,v_7],[v_1,v_4,$

\hspace{1cm}$v_3,v_6,v_5,v_2,v_9,v_{10},v_7,v_8],[v_1,v_5,v_6,v_3,v_4,v_2,v_8,v_7,v_{10},v_9],[v_1,v_8,v_2,v_4],$

\hspace{1cm}$[v_1,v_{10},v_2,v_6],[v_1,v_9,v_2,v_5],[v_1,v_7,v_2,v_3]\}$

is an $\rm SOCDC$ of $G$.

\end{description}
}\end{proof}

\begin{cor}\label{SOCDCK'3}
The minimal counterexample to the $\rm SOCDC$ conjecture is
$3$-edge-connected.
\end{cor}
An edge cut  $F$, is called {\sf trivial} if one of the component in $G\setminus F$ be an isolated vertex.

\begin{thm}
The minimal counterexample to the $\rm SOCDC$ conjecture has no non-trivial edge cut of size  $3$.
\end{thm}
\begin{proof}{
 Let $G$ be the minimal counterexample to the $\rm SOCDC$ conjecture.
We know that $G$ is $2$-connected and $3$-edge-connected.
 Assume that $G$ has a non-trivial edge cut of size  $3$. We consider the following cases.
\item{\rm (I)} {  $G=G_1\cup G_2\cup\{u_1v_1,u_2v_2,u_3v_3\},$
where $G_1\cap G_2=\emptyset$, the vertices $u_i$ are distinct vertices of $G_1$, and the vertices $v_i$
are distinct vertices of $G_2$, $i=1,2,3$}.

Denote by $H_i$ the graph obtained by contracting the
subgraph $G_{i+1}$ to a single vertex $w_i$, $i=1,2$, where subscripts are reduced modulo $2$.
Since $\deg(w_i)=3$, $H_i\neq K_6$, $i=1,2$. By the minimality of $G$, $H_i$ has an $\rm SOCDC$ or
$H_i=K_4$.
Therefore, $H_i$ has an $\rm OCDC$, $\mathcal{C}_i$, $i=1,2$. Let $C_i^j,\ j=1,2,3,$ be the three directed cycles
in $\mathcal{C}_i$ which include $w_i,\ i=1,2$, where without loss of generality, we assume that
$C_1^j$ includes directed path $(u_{j-1},w_1,u_{j+1})$, and
$C_2^j$ includes directed path $(v_{j+1},w_2,v_{j-1})$, where subscripts are reduced modulo $3$, $j=1,2,3$.
Let $P_i^j=C_i^j\setminus w_i,\ i=1,2,\ j=1,2,3$. Define $C^j=P_1^j\cup P_2^j\cup \{u_{j-1}v_{j-1},v_{j+1}u_{j+1}\},$
$\mathcal{C}^{'}=\{C^j : j=1,2,3\},$
and
$\mathcal{C}^{''}=\{C_i^j : i=1,2,\ j=1,2,3\}.$
Thus,
$\mathcal{C}=\mathcal{C}_1\cup\mathcal{C}_2\cup\mathcal{C}^{'}\setminus \mathcal{C}^{''}$
is an $\rm OCDC$ of $G$, where
$|\mathcal{C}|=|\mathcal{C}_1|+|\mathcal{C}_2|-3.$
Note that every $\rm OCDC$ of $K_4$ has $4$ cycles, therefore, in both cases $|\mathcal{C}|\le |V(G)|-1$, which is a contradiction.
\item{\rm (II)} {  $G=G_1\cup G_2\cup\{u_1v_1,u_1v_2,u_2v_3\},$
where $G_1\cap G_2=\emptyset$, the vertices $u_1$ and $u_2$ are distinct vertices of $G_1$, and the vertices $v_i$
are distinct vertices of $G_2$, $i=1,2,3$}.

Denote by $H_i$ the graph obtained by contracting the
subgraph $G_{i+1}$ to a single vertex $w_i$, $i=1,2$, and removing the multiple edge in $H_1$, where subscripts are reduced modulo $2$.
Since $\deg(w_i)=2$ or $3$,  $H_1\neq K_4$ and $H_i\neq K_6$, $i=1,2$. By the minimality of $G$, $H_i$ has an $\rm SOCDC$ or
$H_2=K_4$.
Therefore, $H_i$ has an $\rm OCDC$, $\mathcal{C}_i$, $i=1,2$. Let $C_1^1$ and $C_1^2$ be two directed cycles in $\mathcal{C}_1$
which include $w_1$, where without loss of generality, we assume that
$C_1^j$ includes directed path $(u_{j},w_1,u_{j+1})$, where subscripts are reduced modulo $2$, $j=1,2$, and
 $C_2^k,\ k=1,2,3,$ be the three directed cycles
in $\mathcal{C}_2$ which include $w_2$, where without loss of generality, we assume that
$C_2^k$ includes directed path $(v_{k},w_2,v_{k-1})$, where subscripts are reduced modulo $3$, $k=1,2,3$.
Let $P_1^j=C_1^j\setminus w_1,\ j=1,2,$ and $P_2^k=C_2^k\setminus w_2,\ k=1,2,3$.
Define $C^1=P_1^1\cup P_2^3\cup \{u_{1}v_{2},v_{3}u_{2}\},$
 $C^2=P_1^2\cup P_2^1\cup \{u_{2}v_{3},v_{1}u_{1}\},$ and $C^3=P_2^2\cup \{u_{1}v_{1},v_{2}u_{1}\}$. Let
$\mathcal{C}^{'}=\{C^1,C^2,C^3\},$
and
$\mathcal{C}^{''}=\{C_1^1,C_1^2,C_2^1,C_2^2,C_2^3\}.$
Thus,
$\mathcal{C}=\mathcal{C}_1\cup\mathcal{C}_2\cup\mathcal{C}^{'}\setminus \mathcal{C}^{''}$
is an $\rm OCDC$ of $G$, where
$|\mathcal{C}|=|\mathcal{C}_1|+|\mathcal{C}_2|-2.$
Note that every $\rm OCDC$ of $K_4$ has $4$ cycles, therefore, in both cases $|\mathcal{C}|\le |V(G)|-1$, which is a contradiction.

\item{\rm (III)} {  $G=G_1\cup G_2\cup\{u_1v_1,u_1v_2,u_2v_2\},$
where $G_1\cap G_2=\emptyset$, the vertices $u_1$ and $u_2$ are distinct vertices of $G_1$, and the vertices $v_1$ and $v_2$
are distinct vertices of $G_2$}.

Denote by $H_i$ the graph obtained by contracting the
subgraph $G_{i+1}$ to a single vertex $w_i$, $i=1,2$, and removing the multiple edges, where subscripts are reduced modulo $2$.
Since $\deg(w_i)=2$,  $H_i\neq K_4$ or $ K_6$, $i=1,2$. By the minimality of $G$, $H_i$ has an $\rm SOCDC$, $\mathcal{C}_i$, $i=1,2$.

Let $C_i^j,\ j=1,2,$ be the two directed cycles
in $\mathcal{C}_i$ which include $w_i,\ i=1,2$, where without loss of generality, we assume that
$C_1^j$ includes directed path $(u_{j},w_1,u_{j+1})$, and
$C_2^j$ includes directed path $(v_{j},w_2,v_{j+1})$, where subscripts are reduced modulo $2$, $j=1,2$.
Let $P_i^j=C_i^j\setminus w_i,\ i=1,2,\ j=1,2$. Define $C^1=P_1^1\cup P_2^2\cup \{u_{1}v_{1},v_{2}u_{2}\},$
 $C^2=P_1^2\cup \{u_{2}v_{2},v_{2}u_{1}\},$ and $C^3=P_2^1\cup \{u_{1}v_{2},v_{1}u_{1}\}$. Let
$\mathcal{C}^{'}=\{C^1,C^2,C^3\},$
and
$\mathcal{C}^{''}=\{C_1^1,C_1^2,C_2^1,C_2^2\}.$
Thus,
$\mathcal{C}=\mathcal{C}_1\cup\mathcal{C}_2\cup\mathcal{C}^{'}\setminus \mathcal{C}^{''}$
is an $\rm OCDC$ of $G$, where
$|\mathcal{C}|=|\mathcal{C}_1|+|\mathcal{C}_2|-1.$
 Therefore,  $|\mathcal{C}|\le |V(G)|-1$, which is a contradiction.
}\end{proof}

The properties of the minimal counterexample to the $\rm OPPDC$ conjecture is studied in~\cite{BaOm:OPPDC,MR2046642}, and it is shown that this counterexample is a $2$-connected and $3$-edge-connected graph with minimum degree at least $4$. Regarding to the relation between existence of an $\rm OPPDC$ for a graph and an $\rm SOCDC$ (Theorem~\ref{n-1}), the following relation between the order and the number of edges  of these two minimal counterexamples can be obtained.

Assume that $G_p$ and $G_c$ are the minimal counterexamples to the $\rm OPPDC$ conjecture
and the $\rm SOCDC$ conjecture, respectively. Let $G_p'$ be a graph obtained from $G_p$
by joining a new vertex to all vertex of $G_p$. By Theorem~\ref{n-1}, $G_p'$ has no
$\rm SOCDC$. Note that $G_p$ is connected, so $G_p'$ is $2$-connected. Thus,
$G_p'$ is a counterexample to the $\rm SOCDC$ conjecture. Therefore by the minimality of $G_c$,
$$|V(G_c)|+|E(G_c)|\le|V(G_p')|+|E(G_p')|=2|V(G_p)|+|E(G_p)|+1.$$
Since $\delta(G_c)\ge 3$ ,
$$|V(G_c)|\le \frac{2}{5}( 2|V(G_p)|+|E(G_p)|+1).$$
\section{\large $\rm SOCDC$ and   the Cartesian product}\label{sec:CP}
In~\cite{productI} infinite classes of graphs with an $\rm SCDC$ are obtained using the Cartesian product of graphs.
In this section, we proved the similar results in the oriented version.

The {\sf Cartesian product} of two graphs $G$ and $H$, denoted by $G\square H$,
is the graph with vertex set $V(G)\times V(H)$ and two vertices $(u,v)$
and $(x,y)$ are  adjacent if and only if either $u=x$  and  $vy\in E(H)$
  or  $ux\in E(G)$  and $v=y$.
\begin{thm}\label{GP}
If $G$ has an $\rm OPPDC$, then $G\square P_2$ has
an $\rm SOCDC$ with $|V(G)|$ directed cycles. Furthermore,
if $G$ has an $\rm SOCDC$, then $G\square P_n$, $n\ge3$, has
an $\rm SOCDC$ with at most $|V(G\square P_n)|-1$ directed cycles.
\end{thm}
\begin{proof}{  The proof is similar to the proof of Theorem 1 in~\cite{productI}. 
}\end{proof}

Now we need a theorem about the existence of $\rm OPPDC$ for the Cartesian product of graphs.
\begin{lem}\label{GsquareH}~{\rm \cite{BaOm:OPPDC}}
If  $G$ and $H$ have an $\rm OPPDC$, then $G\square H$ also has an
$\rm OPPDC$.
\end{lem}
The following  corollary follows directly from Theorems~\ref{GP} and~\ref{GsquareH}.
\begin{cor}
If $G$ has an $\rm OPPDC$, then for all $l\ge 2$, $G^l\square P_{2}$
has an $\rm SOCDC$, where $G^l=\overbrace{G\square \cdots\square G}^{l\ times}$.
\end{cor}
\begin{cor}
Every hypercube graph $Q_n,\ n\ge 2,$ has an $\rm SOCDC$.
\end{cor}

With the similar argument as above,  the following theorems, which are the oriented version of  some results in~\cite{productI} for  $\rm SCDC$, can be proved.
%
\begin{thm}\label{GT}
If $G$ has an $\rm OPPDC$ and  an $\rm SOCDC$, then for any tree, $T$, $G\square T$ has an $\rm SOCDC$.
\end{thm}
%
%
%
\begin{thm}\label{GC}
If $G$ has an $\rm OPPDC$, then for all $k\ge 2$, $G\square C_{2k}$ has an $\rm SOCDC$.
Furthermore, if $G$ has an $\rm SOCDC$, then $G\square C_{2k-1}$ has an $\rm SOCDC$.
\end{thm}

The following corollary  concludes  directly from Theorems~\ref{GsquareH} and~\ref{GC}.
\begin{cor}
If $G$ has an $\rm OPPDC$, then for all $k,l\ge 2$, $G^l\square C_{2k}$ has
an $\rm SOCDC$, where $G^l=\overbrace{G\square \cdots\square G}^{l\ times}$.
\end{cor}
\begin{thm}
 If $G$ has an $\rm SOCDC$,
then for  $n\ge 2|V(G)|+1$, $G\square C_n$ has an $\rm SOCDC$.
\end{thm}
\begin{proof}{
Let $\mathcal{F}$ be the union of the corresponding $\rm SOCDC$'s of copies of $G$  and the corresponding $\rm SOCDC$'s of copies of $C_n$. Since every $\rm OCDC$ of $C_n$ has two directed cycles,
and $n\ge 2|V(G)|+1$,
we have $|\mathcal{F}|\le n(|V(G)|-1)+2|V(G)|\le n(|V(G)|-1)+(n-1)\le n|V(G)|-1=|V(G\square C_n)|-1.$
Therefore, $\mathcal{F}$ is an  $\rm SOCDC$ of $G\square C_n$.
}\end{proof}

\section*{Acknowledgments}
The authors thank Professor C.Q. Zhang
for reading the manuscript and for his  helpful suggestions.

\bibliographystyle{plain}

\end{document}